\documentclass[11pt]{amsart}
\usepackage{amsmath}

\usepackage{hyperref}
\hypersetup{colorlinks=true, urlcolor=black, linkcolor=black, citecolor=black}

\usepackage{dsfont}

\newtheorem{theorem}{Theorem}
\newtheorem{lemma}{Lemma}
\newtheorem{corollary}{Corollary}

\begin{document}

\title[Divergent spherical harmonic expansions]{Almost everywhere divergence of spherical harmonic expansions and equivalence of summation methods}
\author{Xianghong Chen \and Dashan Fan \and Juan Zhang}
\address{Xianghong Chen\\Department of Mathematical Sciences\\University of Wisconsin-Milwaukee\\Milwaukee, WI 53211, USA}
\email{chen242@uwm.edu}
\address{Dashan Fan\\Department of Mathematical Sciences\\University of Wisconsin-Milwaukee\\Milwaukee, WI 53211, USA}
\email{fan@uwm.edu}
\address{Juan Zhang\\School of Mathematical Sciences\\Beijing Normal University\\Laboratory of Mathematics and Complex Systems\\Beijing 100875,  People's Republic of China}
\email{zhangjuanw@mail.bnu.edu.cn}
\date{\today}

\subjclass[2010]{43A85, 43A50, 40G05, 33C55} 
\keywords{Ces\`aro means, Bochner-Riesz means, spherical harmonic expansion, almost everywhere divergence} 

\maketitle

\begin{abstract}
We show that there exists an integrable function on the $n$-sphere $(n\ge 2)$, whose Ces\`aro (C,$\frac{n-1}{2}$) means with respect to the spherical harmonic expansion diverge unboundedly almost everywhere. By studying equivalence theorems, we also obtain the corresponding results for Riesz and Bochner-Riesz means. 
This extends results of Stein (1961) for flat tori and complements the work of Taibleson (1985) for spheres. 
\end{abstract}

\tableofcontents

\section{Introduction}\label{sec:introduction}

Let $\mathbb S^n$ ($n\ge 2$) be the $n$-dimensional sphere equipped with the uniform measure. We consider for $f\in L^1(\mathbb S^n)$ the
 spherical harmonic expansion
\begin{equation}\label{intro:expansion}
f\sim \sum_{k=0}^\infty \text{proj}_k f,
\end{equation}
where $\text{proj}_k$ denotes the orthogonal projection operator from $L^2(\mathbb S^n)$ to the space of spherical harmonics of degree $k$ (cf. \cite[Chapter~IV, Section~2]{SteinWeiss1971}). It is well known that the series  in \eqref{intro:expansion} diverges for general $f\in L^1(\mathbb S^n)$. So it is natural to consider summation methods that guarantee the convergence of \eqref{intro:expansion}. In this regard, we consider the Ces\`aro (C,$\delta$) means, $S^\delta_N f,\ N=0, 1, \cdots$, for the series in \eqref{intro:expansion}. Bonami and Clerc \cite{BonamiClerc1973} showed that $S_N^\delta f$ converges almost everywhere to $f$ provided $\delta>\delta_0:=\frac{n-1}{2}$. For $\delta<\delta_0$, Meaney \cite{Meaney2003} showed that there exists a zonal function $f\in L^1(\mathbb S^n)$ 
such that $S_N^\delta f$ diverges almost everywhere. At the critical index $\delta=\delta_0$, a general result of Christ and Sogge \cite{ChristSogge1988} implies 
that $S_N^{\delta_0} f$ always converges in measure to $f$. Chanillo and Muckenhoupt \cite{ChanilloMuckenhoupt1993} showed that, if moreover $f$ is zonal, then $S_N^{\delta_0} f$ converges almost everywhere to $f$.

In view of Kolmogoroff's counterexample \cite{Kolmogoroff1923} on $\mathbb S^1$ 
and Stein's counterexample \cite{Stein1961} on the tori $\mathbb T^n$ ($n\ge 2$), 
it can be expected that there exists an $f\in L^1(\mathbb S^n)$ such that $S_N^{\delta_0} f$ diverges almost everywhere on $\mathbb S^n$. Indeed, such a result is claimed in a paper by Taibleson \cite{Taibleson1985} with $f$ belonging to the Hardy space $H^1(\mathbb S^n)$ (note that $H^1\subset L^1$; see also remarks in \cite{ColzaniTaiblesonWeiss1984}, \cite{Zaloznik1988}). However, Taibleson addresses only how to obtain such a result from the corresponding result with $f\in L^1(\mathbb S^n)$, by adapting an idea of Stein in \cite{Stein1983} for $\mathbb T^n$. The main purpose of this paper is to complete the proof of Taibleson's claim by constructing an $f\in L^1(\mathbb S^n)$ such that $S_N^{\delta_0} f$ diverges almost everywhere (see Section \ref{sec:divergence}, Theorem \ref{thm:cesaro-divergence}). 
Although our proof bears some similarity to those for $\mathbb S^1$ and $\mathbb T^n$, some aspects are new and may provide motivation for the study of more general settings, and of more refined questions. 

Another purpose of this paper is to study the relations among Ces\`aro, Riesz, and Bochner-Riesz means. These summation methods are known to produce the same order of summability for any numerical series (cf. \cite{HardyRiesz1964}). 
However, the situation is more delicate at the critical order if one considers divergence properties.\footnote{By the word ``divergence'' we will always mean unbounded divergence.} In Section \ref{sec:summation}, we use a result of Ingham \cite{Ingham1968/1969} to show that the three summation methods are equidivergent in an almost everywhere sense  
when $\delta=\delta_0$ and $f\in L^1(\mathbb S^n)$. In particular, we obtain from the almost everywhere divergence result for Ces\`aro means almost everywhere divergence results for Riesz and Bochner-Riesz means. 

The construction of $f$ relies on precise estimates of the summation kernel. In the case of tori \cite{Stein1961} and compact semisimple Lie groups  
\cite{ChenFantoappear}, such estimates were obtained using the Poisson summation formula. However, except for some special cases, 
such an approach does not seem to carry over to $\mathbb S^n$. Based on detailed study of the Jacobi polynomials (cf. \cite{Szegoe1975}), Bonami and Clerc \cite{BonamiClerc1973} were able to obtain rather precise estimates of the Ces\`aro kernels on $\mathbb S^n$. They applied the estimates mainly with $\delta>\delta_0$ in \cite{BonamiClerc1973}. For the sake of self-containedness, in Section \ref{sec:notation} we give a detailed presentation for the case $\delta=\delta_0$ following their approach. 

With the kernel estimates, to obtain the almost everywhere divergence result, 
we combine ideas of Stein \cite{Stein1961} in his treatments of $\mathbb S^1$ and resp. $\mathbb T^n$, as well as the treatment of compact semisimple Lie groups by the first two authors in \cite{ChenFantoappear}. More precisely, as in \cite{ChenFantoappear} we first use Young's inequality and the weak (1,1) boundedness of the Hardy-Littlewood maximal function 
to remove the influence of the global part of the kernel. 
We then need to find an appropriate $f\in L^1(\mathbb S^n)$ to blow up the local part of the kernel. For this we use an idea of Stein in his treatment of $\mathbb S^1$, that is to replace $f$ by an appropriate probability measure $\mu$ whose mass is equally distributed on finitely many points. The points in the support of $\mu$ need to be suitably equidistributed, and the distance functions generated by them need to be rationally independent almost everywhere. In the case of compact semisimple Lie groups, such points were constructed in \cite{ChenFantoappear} using a probabilistic approach similar to that of Kahane \cite{Kahane1960/1961} for $\mathbb S^1$. 
However, the probabilistic approach is somewhat limited and provides limited information about details of the divergence. Therefore we opt for a deterministic approach for $\mathbb S^n$. As the proofs will show, this approach provides much more flexibility in choosing the points. In particular, it will be shown that the equidistribution property is satisfied for any sufficiently dense packing of $\mathbb S^n$ (or modifications thereof; see Lemma \ref{lem:Riemann-sum}). 
It will also be shown that the rational independence property holds whenever the points are distinct and contain no antipodal pairs (Lemma \ref{lem:point-masses}). 
It is worth mentioning that the latter requires knowledge about the analyticity of the distance functions on $\mathbb S^n$ (see \cite{Bochner1936}, \cite{Stein1961} for the case of $\mathbb R^n$); in particular, it is used in the proof that the cut loci of $\mathbb S^n$ are singletons.
Interestingly, similar considerations are not needed in the case of $\mathbb T^n$. This is because the global part of the kernel on $\mathbb T^n$ already diverges almost everywhere, so that one can simply take $\mu$ to be a point mass, see \cite{Stein1961}. 

Throughout the paper, unless otherwise stated, $C$ denotes a positive, dimensional constant whose value may change from line to line. $A=O(B)$ means that $|A|\le \widetilde C B$ holds for a constant $\widetilde C>0$ independent of the testing inputs (which will usually be clear from the context). 

\section{Notation and preliminaries}\label{sec:notation}

In what follows, we denote by $\mathbb S^{n}$ the unit sphere in $\mathbb R^{n+1}$ ($n\ge2$) equipped with the standard round metric. Denote by $\Delta$ the Laplace-Beltrami operator on $\mathbb S^{n}$. For $k=0,1,\cdots$, denote by
$\mathcal H_k^n$ the space of spherical harmonics of degree $k$ (for background on spherical harmonics, cf. \cite[Chapter IV]{SteinWeiss1971}, \cite{WangLi2000}, \cite{DaiXu2013}). It is well known that we have the  orthogonal decomposition 
$$L^2(\mathbb S^n)=\bigoplus_{k=0}^\infty \mathcal H_k^n;$$
moreover, 
\begin{equation}\label{eq:eigenvalue}
\Delta Y_k=-k(k+n-1)\,Y_k,\ \forall Y_k\in \mathcal H_k^n,
\end{equation}
$$\dim\mathcal H_k^n=\binom{k+n}{k}-\binom{k-2+n}{k-2}.$$
Denote by
$$\text{proj}_k: L^2(\mathbb S^n)\rightarrow \mathcal H_k^n$$
the orthogonal projection from $L^2(\mathbb S^n)$ to $\mathcal H_k^n$. Set
$$\lambda=\frac{n-1}{2}.$$
Let $C_k^\lambda(t)$ be the Gegenbauer polynomial of degree $k$ and index  $\lambda$. Equivalently,
\begin{equation}\label{eq:gegenbauer-jacobi}
C_k^\lambda(t)=\frac{\Gamma(\lambda+\frac{1}{2})}{\Gamma(2\lambda)}
\frac{\Gamma(k+2\lambda)}{\Gamma(k+\lambda+\frac{1}{2})} 
P_k^{(\lambda-\frac{1}{2},\lambda-\frac{1}{2})}(t),
\end{equation}
where $P_k^{(\alpha,\beta)}$ is the Jacobi polynomial of degree $k$ (cf. Szeg\"o \cite[p.~80]{Szegoe1975}). 
Denote by
$$|x-y|\in [0,\pi]$$
the great-circle distance between $x$ and $y$ on $\mathbb S^n$. It is a standard fact that
\begin{equation}\label{eq:proj-zonal}
\text{proj}_k f=\int_{\mathbb S^n} Z_k(\cdot,y) f(y)dy,
\end{equation}
where
\begin{equation}\label{eq:zonal-def}
Z_k(x,y)=\frac{k+\lambda}{\lambda}C_k^\lambda \big(\cos|x-y|\big).
\end{equation}
Note that 
$$\cos|x-y|=\langle x,y \rangle$$
represents the inner product in $\mathbb R^{n+1}$. Note also that, it follows from \eqref{eq:gegenbauer-jacobi} and Szeg\"o \cite[Theorem~7.31.2]{Szegoe1975}
\begin{equation}\label{eq:Z-bound} 
|Z_k(x,y)| 
\le C (k+1)^\frac{n-1}{2} 
|x-y|^{-\frac{n-1}{2}} |x-\hat y|^{-\frac{n-1}{2}},\ k=0,1,\cdots,
\end{equation}
where $\hat y$ denotes the antipodal point of $y$ on $\mathbb S^n$. (See also Bonami and Clerc \cite[p.~231]{BonamiClerc1973}.)

For $f\in L^1(\mathbb S^n)$, we can consider the formal spherical harmonic expansion
$$f\sim \sum_{k=0}^\infty \text{proj}_k f,$$
with $\text{proj}_k f\in\mathcal H^n_k$ given by \eqref{eq:proj-zonal}. 
Let $\delta>-1$. The corresponding Ces\`{a}ro (C,$\delta$) means are defined by 
\begin{equation}\label{eq:cesaro-mean-def}
S^\delta_N f = \frac{1}{A_N^\delta}\sum_{k=0}^{N} A_{N-k}^\delta \text{proj}_k f,\ N=0, 1, \cdots
\end{equation}
where
\begin{equation}\label{eq:A-delta-k}
A_k^\delta=\binom{k+\delta}{k}=\frac{(\delta+k)(\delta+k-1)\cdots(\delta+1)}{k(k-1)\cdots 1}=\frac{\Gamma(k+\delta+1)}{\Gamma(k+1)\Gamma(\delta+1)}.
\end{equation}
Note that the sequence $\{A_k^\delta\}_{k=0}^\infty$ satisfies
\begin{equation}\label{eq:taylor}
(1-x)^{-1-\delta}=\sum_{k=0}^\infty A_k^\delta x^k,\quad -1<x<1.
\end{equation}
By \eqref{eq:proj-zonal}, we can also write
\begin{equation}\label{eq:cesaro-sum-with-kernel}
S^\delta_N f=\int_{\mathbb S^n}K^\delta_N(\cdot,y)f(y)dy
\end{equation}
where
\begin{equation}\label{eq:cesaro-kernel-def}
K^\delta_N(x,y)=\frac{1}{A_N^\delta}\sum_{k=0}^{N} A_{N-k}^\delta Z_k(x,y)
\end{equation}
is called the ($N$-th) Ces\`{a}ro kernel of order $\delta$. Note that $K^\delta_N(x,y)$ depends only on $|x-y|$. In general, for kernels $K(x,y)$ satisfying this property, we will write
\begin{equation}\label{eq:kernel-convolution}
K*f:=\int_{\mathbb S^n} K(\cdot,y)f(y)dy.
\end{equation}
In particular, we will write
$$S^\delta_N f=K^\delta_N*f.$$

In what follows, we will focus on the case
$$\delta=\delta_0=\frac{n-1}{2}.$$
For simplicity, we will write
\begin{equation}\label{eq:kernel-no-delta}
K_N(x,y)=K^{\delta_0}_N(x,y).
\end{equation}
Following Bonami and Clerc \cite{BonamiClerc1973}, we first decompose $K_N(x,y)$ as
$$K_N(x,y)=K_N^{(0)}(x,y)+K_N^{(\pi)}(x,y)$$
where
\begin{align}
K_N^{(0)}(x,y)& =K_N(x,y)\,\mathds 1_{\{|x-y|\le \pi/2\}},\label{eq:K-(0)}\\
K_N^{(\pi)}(x,y)& =K_N(x,y)\,\mathds 1_{\{|x-y|>\pi/2\}}\label{eq:K-(pi)}.
\end{align}
We will need the following estimate for $K_N^{(\pi)}(x,y)$. See \cite[Corollary~2.5]{BonamiClerc1973}.

\begin{lemma}\label{lem:antipodal}
There exists a constant $C>0$, such that
$$|K_N^{(\pi)}(x,y)|\le C|x-\hat y|^{-\frac{n-1}{2}},\ N=0,1,\cdots.$$
\end{lemma}

To estimate $K_N^{(0)}(x,y)$, following Bonami and Clerc \cite{BonamiClerc1973}, we will use the following formula (see Szeg\"{o} \cite[p.~261]{Szegoe1975}).

\begin{lemma}\label{lem:jacobi}
For $N=1,2,\cdots$, we have
$$K_N(x,y)=C_N P_N^{(n-\frac{1}{2},\frac{n-2}{2})}(\cos|x-y|)+E_N(x,y)$$
where
\begin{align*}
& C_N=\frac{1}{A^{\delta_0}_N}
\left\{ 2^{n-1} \Gamma\big(\frac{n}{2}\big)
\frac{\Gamma\big(N+\frac{n}{2}\big)}
{\Gamma\big(N+\frac{3n}{2}-1\big)}
\frac{\Gamma\big(2N+2n-1\big)}
{\Gamma\big(2N+\frac{3n}{2}\big)}\right\}^{-1},\\
& E_N(x,y)=-\sum_{\ell=1}^{\infty} (-1)^\ell\, 
\binom{\delta_0}{\ell}
\frac{(N+\frac{n+1}{2})\cdots(N+\frac{n+1}{2}+\ell-1)}{(2N+\frac{3n}{2})\cdots(2N+\frac{3n}{2}+\ell-1)}
K^{\delta_0+\ell}_N(x,y).
\end{align*}
\end{lemma}

Regarding the coefficients in Lemma \ref{lem:jacobi}, we have
\begin{equation}\label{eq:C_N-asympt}
\lim_{N\rightarrow\infty} \frac{C_N}{N^{1/2}}=2^{-\frac{3n-4}{2}}\frac{\Gamma\Big(\frac{n+1}{2}\Big)}{\Gamma\Big(\frac{n}{2}\Big)}, 
\end{equation}
and 
\begin{equation}\label{eq:abs-conv}
\sum_{\ell=1}^{\infty} \left |\binom{\delta_0}{\ell}
\frac{(N+\frac{n+1}{2})\cdots(N+\frac{n+1}{2}+\ell-1)}{(2N+\frac{3n}{2})\cdots(2N+\frac{3n}{2}+\ell-1)}\right |
\le \sum_{\ell=1}^{\infty} \left |\binom{\delta_0}{\ell}\right |
<\infty,
\end{equation}
where the last series converges because
$$\left|\binom{\delta_0}{\ell}\right|
=\left |\frac{\delta_0(\delta_0-1)\cdots(\delta_0-\ell+1)}{1\cdot 2\cdots\ell}\right |
\le C \ell^{-1-\delta_0}.$$

We will need the following asymptotic formula for the Jacobi polynomial (cf. Szeg\"{o} \cite[Theorem~8.21.8]{Szegoe1975}).

\begin{lemma}\label{lem:jacobi-poly-asympt}
For every $t=\cos |x-y|\in (-1,1)$, 
$$P_N^{(n-\frac{1}{2},\frac{n-2}{2})}(t)
=N^{-\frac{1}{2}} k(x,y)\cos\left (\Big(N+\frac{3n-1}{4}\Big)|x-y|-\frac{n\pi}{2}\right )+O(N^{-\frac{3}{2}})$$
where
\begin{equation}\label{eq:k(x,y)}
k(x,y)=\pi^{-\frac{1}{2}} \left( \sin \frac{|x-y|}{2} \right)^{-n}
\left( \sin \frac{|x-\hat y|}{2} \right)^{-\frac{n-1}{2}}.
\end{equation}
\end{lemma}

To estimate $E_N(x,y)$ in Lemma \ref{lem:jacobi}, we will use the following observation (see also \cite[Lemma~2.4]{BonamiClerc1973} for a related result). 

\begin{lemma}\label{lem:major-a}
Let $\{a_k\}_{k=0}^{\infty}$ be a sequence of numbers and let $\delta>-1$. Denote
$$S^\delta_N = \frac{1}{A_N^\delta}\sum_{k=0}^{N} A_{N-k}^\delta a_k.$$
Then for any $\rho>0$,
\begin{equation}\label{eq:major-S-rho}
|S^{\delta+\rho}_N| \le \max_{0\le k\le N} |S^{\delta}_k|,\ N=0, 1, \cdots.
\end{equation}
\end{lemma}

\begin{proof}
By writing
$$(1-x)^{-1-\delta-\rho}=(1-x)^{-1-(\rho-1)}(1-x)^{-1-\delta}$$
and using \eqref{eq:taylor}, we see that
\begin{equation}\label{eq:summation}
A_{N}^{\delta+\rho}=\sum_{\ell=0}^{N} A_{\ell}^{\rho-1} A_{N-\ell}^\delta.
\end{equation}
Thus, we can write
\begin{align*}
S^{\delta+\rho}_N
&=\frac{1}{A_N^{\delta+\rho}}\sum_{k=0}^{N} A_{N-k}^{\delta+\rho} a_k\\
&=\frac{1}{A_N^{\delta+\rho}}\sum_{k=0}^{N} \left (\sum_{\ell=0}^{N-k} A_{\ell}^{\rho-1} A_{N-k-\ell}^\delta \right ) a_k\\
&=\frac{1}{A_N^{\delta+\rho}}\sum_{\ell=0}^{N} A_{\ell}^{\rho-1} {A_{N-\ell}^\delta} \left (\frac{1}{A_{N-\ell}^\delta}\sum_{k=0}^{N-\ell} A_{N-\ell-k}^\delta a_k\right ),
\end{align*}
that is,
\begin{equation}\label{eq:summation-S}
S^{\delta+\rho}_N=\frac{1}{A_N^{\delta+\rho}}\sum_{\ell=0}^{N}  A_{\ell}^{\rho-1} {A_{N-\ell}^\delta} S_{N-\ell}^\delta.
\end{equation}
Now bounding 
$$|S_{N-\ell}^\delta|\le \max_{0\le k\le N} |S^{\delta}_k|,\ 0\le\ell\le N,$$
and using \eqref{eq:summation} again, the inequality \eqref{eq:major-S-rho} follows.
\end{proof}

We will also need the following estimates for $K_N^{\delta_0+1}$. See \cite[Corollary~2.5]{BonamiClerc1973}.

\begin{lemma}\label{lem:K^{delta+1}}
(i) When $|x-y|\le \pi/2$,
$$|K_N^{\delta_0+1}(x,y)|\le C N^n,\ N=1,2,\cdots.$$
(ii) When $2/N\le |x-y|\le \pi/2$,
$$|K_N^{\delta_0+1}(x,y)|\le C N^{-1} |x-y|^{-n-1}.$$
\end{lemma}

Let $\nu$ be a finite Borel measure on $\mathbb S^n$. Denote by 
\begin{equation}\label{eq:hardy-littlewood-def}
M(\nu)(x)=\sup_{r>0}\frac{|\nu|(B(x,r))}{|B(x,r)|},\ x\in\mathbb S^n
\end{equation}
the Hardy-Littlewood maximal function of $\nu$ on $\mathbb S^n$, where
$$B(x,r)=\big\{y\in \mathbb S^n: |x-y|<r\big\}.$$
Note that, for some constant $C>1$, 
\begin{equation}\label{eq:AD-regularity}
r^n/C\le |B(x,r)|\le C r^n,\ 0<r\le \pi
\end{equation}
Using a standard argument (cf. \cite[Theorem~2.2]{Heinonen2001}), we have the weak $(1,1)$ inequality 
\begin{equation}\label{eq:weak-1-1}
\big|\{x\in\mathbb S^n: M(\nu)(x)>t\}\big|\le C\frac{\|\nu\|}{t},\ t>0.
\end{equation}
In particular, by taking $t\rightarrow\infty$, we have
$$M(\nu)(x)<\infty,\ \text{a.e. } x\in\mathbb S^n.$$
Using Lemma \ref{lem:K^{delta+1}} and a dyadic decomposition, we also have 
\begin{equation}\label{eq:maximal-K^{delta+1}}
\sup_{N\ge 0} \left |\int_{|x-y|\le \pi/2} K_N^{\delta_0+1}(x,y)\,d\nu(y)\right |
\le C M(\nu)(x),\ x\in\mathbb S^n.
\end{equation}

Now, following the notations in Lemma \ref{lem:jacobi}, let us denote
\begin{align}
&\widetilde K_N^{(0)}(x,y)=C_N P_N^{(n-\frac{1}{2},\frac{n-2}{2})}(\cos|x-y|)\,\mathds 1_{\{|x-y|\le \pi/2\}},\label{eq:tilde-K-0}\\
&E_N^{(0)}(x,y)=E_N(x,y)\,\mathds 1_{\{|x-y|\le \pi/2\}}.\label{eq:E-0}
\end{align}
We can then further decompose
\begin{equation}\label{eq:K-decompose-3}
K_N(x,y)=\widetilde K_N^{(0)}(x,y) + E_N^{(0)}(x,y) + K_N^{(\pi)}(x,y),\ N=1,2,\cdots.
\end{equation}
Applying Lemma \ref{lem:major-a} with $\delta=\delta_0+1$, 
$$a_k=\int_{|x-y|\le \pi/2} Z_k(x,y)\,d\nu(y),$$
and $\rho=1,2,\cdots$, we see that for $\ell=1,2,3,\cdots$,
$$\left |\int_{|x-y|\le \pi/2} K_N^{\delta_0+\ell}(x,y)\,d\nu(y)\right |
\le \sup_{N\ge 0} \left |\int_{|x-y|\le \pi/2} K_N^{\delta_0+1}(x,y)\,d\nu(y)\right |.$$
Thus, it follows from \eqref{eq:abs-conv} that
\begin{equation}\label{eq:major-K-ell}
|E_N^{(0)}*\nu(x)|\le C \sup_{N\ge 0} \left |\int_{|x-y|\le \pi/2} K_N^{\delta_0+1}(x,y)\,d\nu(y)\right |,\ N=1,2,\cdots.
\end{equation}
Combining this with Lemma \ref{lem:antipodal} and \eqref{eq:maximal-K^{delta+1}}, we obtain the desired estimate.

\begin{lemma}\label{lem:decompose}
There exists a constant $C>0$, such that for any finite Borel measure $\nu$ on $\mathbb S^n$,
$$ \big|(K_N-\widetilde K_N^{(0)})*\nu\big|\le C M(\nu)+C k^{(\pi)}*|\nu|,\ N=1,2,\cdots,$$
where
$$k^{(\pi)}(x,y)=|x-\hat y|^{-\frac{n-1}{2}}.$$
\end{lemma}

\section{Almost everywhere divergence of Ces\`{a}ro means}\label{sec:divergence}

With Lemma \ref{lem:decompose}, we can now prove the main theorem. 

\begin{theorem}\label{thm:cesaro-divergence}
There exists a function $f\in L^1(\mathbb S^n)$ such that
$$\limsup_{N\rightarrow\infty}\big|S_N^{\frac{n-1}{2}} f(x)\big|=\infty,\ \text{a.e. } x\in\mathbb S^n.$$ 
\end{theorem}

In a similar spirit as in \cite{Stein1961}, we will deduce Theorem \ref{thm:cesaro-divergence} from the following lemma. Denote by $\mathcal P(\mathbb S^n)$ the set of Borel probability measures on $\mathbb S^n$.

\begin{lemma}\label{lem:tilde-K-mu}
Given $L>1$, there exists a finitely supported measure $\mu\in\mathcal P(\mathbb S^n)$ such that
$$\limsup_{N\rightarrow\infty} 
\big|\widetilde K_N^{(0)}*\mu(x)\big|>L,\ \text{a.e. } x\in\mathbb S^n.$$
\end{lemma}

The proof of Lemma \ref{lem:tilde-K-mu} is postponed to the end of this section. Using Lemma \ref{lem:decompose}, we first show that Lemma \ref{lem:tilde-K-mu} implies the following (note that $\widetilde K_N^{(0)}$ is replaced by $K_N$).

\begin{lemma}\label{lem:K-mu}
Given $L>1$ and $\varepsilon>0$, there exists a finitely supported measure $\mu\in\mathcal P(\mathbb S^n)$ such that
$$\limsup_{N\rightarrow\infty} 
\big|K_N*\mu(x)\big|>L$$
holds on a set $E\subset \mathbb S^n$ with $|\mathbb S^n\backslash E|<\varepsilon$.
\end{lemma}

\begin{proof}
By Lemma \ref{lem:tilde-K-mu}, for any $\widetilde L>1$, we can find a finitely supported measure $\mu\in\mathcal P(\mathbb S^n)$ such that
$$\limsup_{N\rightarrow\infty} 
\big|\widetilde K_N^{(0)}*\mu(x)\big|>\widetilde L$$
holds on a set $\widetilde E\subset \mathbb S^n$ with $|\mathbb S^n\backslash \widetilde E|=0$. It follows that, on this set $\widetilde E$, 
$$\limsup_{N\rightarrow\infty} 
\big|K_N*\mu(x)\big|>\widetilde L/2$$
holds whenever
\begin{equation}\label{eq:error-less-L}
\limsup_{N\rightarrow\infty} 
\big|(K_N-\widetilde K_N^{(0)})*\mu(x)\big|\le \widetilde L/2.
\end{equation}
On the other hand, by Lemma \ref{lem:decompose}, we have
$$\limsup_{N\rightarrow\infty} \big|(K_N-\widetilde K_N^{(0)})*\mu\big|\le C M(\mu)+ C k^{(\pi)}*\mu.$$
Thus, by \eqref{eq:weak-1-1} and Young's inequality, 
$$\big|\{x\in\mathbb S^n: \eqref{eq:error-less-L} \text{ fails}\}\big|\le {C}/{\widetilde L}.$$
So, by choosing $\widetilde L$ large enough that 
$$\widetilde L/2>L,\quad {C}/{\widetilde L}<\varepsilon,$$
and letting
$$E=\widetilde E\,\backslash\{x\in\mathbb S^n: \eqref{eq:error-less-L} \text{ fails}\},$$
the lemma is established. 
\end{proof}

Next, we replace the measure $\mu$ in Lemma \ref{lem:K-mu} by a function $f\in L^1(\mathbb S^n)$. We will call $f$ a polynomial of degree $N$ if 
$$\text{proj}_N f\neq 0\quad\text{and}\quad\text{proj}_k f=0,\ \forall k\ge N+1.$$

\begin{lemma}\label{lem:K-polynomial}
Given $L>1$ and $\varepsilon>0$, there exist a polynomial $f$ with $\|f\|_{L^1(\mathbb S^n)}\le 1$, and an integer $N_0$, such that
$$\max_{0\le N\le N_0} 
\big|K_N*f(x)\big|>L$$
holds on a set $E\subset \mathbb S^n$ with $|\mathbb S^n\backslash E|<\varepsilon$.
\end{lemma}

\begin{proof}
By Lemma \ref{lem:K-mu}, for any $\widetilde L>1$, we can find a finitely supported measure $\mu\in\mathcal P(\mathbb S^n)$ such that
$$\limsup_{N\rightarrow\infty} 
\big|K_N*\mu(x)\big|>\widetilde L$$
holds on a set $\widetilde E\subset \mathbb S^n$ with $|\mathbb S^n\backslash \widetilde E|<\varepsilon/2$. By a standard limiting argument, there exists an integer $N_0$ such that
\begin{equation}\label{eq:K-mu-L/2}
\max_{0\le N\le N_0} 
\big|K_N*\mu(x)\big|>\widetilde L/2
\end{equation}
holds on a set $E\subset \mathbb S^n$ with $|\mathbb S^n\backslash E|<\varepsilon$. Notice that
$$K_N*\mu=\sum_{k=0}^{N} \frac{A_{N-k}^{\delta_0}}{A_N^{\delta_0}} \text{proj}_k \mu.$$
If we let $f=S^{\delta_0+1}_{N_1}\mu$ with $N_1\ge N$, then $f$ is a polynomial of degree at most $N_1$; moreover, 
$$K_N*f=\sum_{k=0}^{N}  \frac{A_{N-k}^{\delta_0}}{A_N^{\delta_0}}\frac{A_{N_1-k}^{\delta_0+1}}{A_{N_1}^{\delta_0+1}} \text{proj}_k \mu.$$
Thus, for $N\le N_0\le N_1$, we can bound
\begin{align*}
\big|K_N*\mu-K_N*f\big|
& \le \sum_{k=0}^{N}  \frac{A_{N-k}^{\delta_0}}{A_N^{\delta_0}}
\left (1-\frac{A_{N_1-k}^{\delta_0+1}}{A_{N_1}^{\delta_0+1}}\right )|\text{proj}_k \mu|\\
& \le \left (1-\frac{A_{N_1-N_0}^{\delta_0+1}}{A_{N_1}^{\delta_0+1}}\right )  \sum_{k=0}^{N_0} \|Z_k\|_\infty.
\end{align*}
Note that the last expression converges to 0 as $N_1\rightarrow\infty$. If we choose $N_1$ large enough, then we obtain from \eqref{eq:K-mu-L/2} that
$$\max_{0\le N\le N_0} 
\big|K_N*f(x)\big|>\widetilde L/4,\ \forall x\in E.$$ 
On the other hand, it follows from Lemmas \ref{lem:antipodal}, \ref{lem:major-a}, \ref{lem:K^{delta+1}}, and Young's inequality that
$$\|f\|_{L^1(\mathbb S^n)}=\|S^{\delta_0+1}_{N_1}\mu\|_{L^1(\mathbb S^n)}\le C.$$
So, by considering the function $f/C$ and choosing $\widetilde L$ large enough that
$$\frac{\widetilde L}{4C}>L,$$
we obtain the desired properties.
\end{proof}

We are now ready to prove Theorem \ref{thm:cesaro-divergence}. 

\begin{proof}[Proof of Theorem \ref{thm:cesaro-divergence}]
The function $f$ will be taken to be of the form
$$f=\sum_{j=0}^\infty \eta_j f_j,$$
where $\{\eta_j\}_{j=0}^\infty\in\ell^1$ is a sequence of positive numbers, and each $f_j$ is a polynomial satisfying $\|f_j\|_{L^1(\mathbb S^n)}\le 1$. Obviously, $$\|f\|_{L^1(\mathbb S^n)}\le \sum_{j=0}^\infty \eta_j<\infty.$$
Each $f_j$ will have an associated integer, denoted by $N_j$. The choice of $\eta_j, f_j, N_j$ is based on induction. More precisely, set $\eta_0=1, f_0=0, N_0=0$. 
Assuming that $\eta_{j-1}$, $f_{j-1}$, $N_{j-1}$ have been chosen, we now describe how we choose $\eta_j, f_j, N_j$. 

First, we choose $\eta_j>0$ small enough so that
\begin{equation}\label{eta<eta}
\eta_j\le\eta_{j-1}/2
\end{equation}
and so that
\begin{equation}\label{etaK<1}
\eta_j \max_{0\le N\le N_{j-1}}
\big\|K_N\big\|_{\infty}\le 1.
\end{equation}
With $\eta_j$ chosen, by Lemma \ref{lem:K-polynomial} we can find a polynomial $f_j$ with $\|f_j\|_{L^1(\mathbb S^n)}\le 1$ and an integer $N_j$, such that
\begin{equation}\label{eqn:sup>L}
\eta_j\max_{0\le N\le N_j} 
\big|K_N*f_j(x)\big|
>\sup_{N\ge 0}\Big \|K_N*(\eta_0 f_0+\cdots+\eta_{j-1} f_{j-1})\Big \|_\infty+j
\end{equation}
holds on a set $E_j\subset \mathbb S^n$ with $|\mathbb S^n\backslash E_j|<j^{-1}$ 
(note that the right-hand side of \eqref{eqn:sup>L} is finite because 
$\eta_0 f_0+\cdots+\eta_{j-1} f_{j-1}$ is a polynomial). By induction, this completes the choice of $\eta_j, f_j, N_j$ for all $j$. 

Now let 
$$E=\bigcap_{i=1}^\infty \bigcup_{j\ge i} E_j.$$ 
It is easy to see that $|\mathbb S^n\backslash E|=0$. To complete the proof, we will show
$$\limsup_{N\rightarrow\infty} 
\big|K_N*f(x)\big|
=\infty,\ \forall x\in E.$$
Fix $x\in E$. By the definition of $E$, there are infinitely many $j$'s for which $x\in E_j$. Fix such an index $j_0$. Then we can write 
\begin{align*}
K_N*f(x)
&=K_N*\Big(\sum_{j<j_0} \eta_j f_j\Big)(x)
+\eta_{j_0} K_N*f_{j_0}(x)
+K_N*\Big(\sum_{j>j_0} \eta_j f_j\Big)(x)\\
&=I+II+III.
\end{align*}
By \eqref{eta<eta} and \eqref{etaK<1}, for $N\le N_{j_0}$ we have
\begin{align*}
|III|
&\le \big\|K_N\big\|_{\infty} \sum_{j>j_0} \eta_j \|f_j\|_{L^1(\mathbb S^n)}\\
&\le \Big (\sum_{j>j_0} \eta_j \Big )
\max_{0\le N\le N_{j_0}}
\big\|K_N\big\|_{\infty}\\
&\le 2\eta_{j_0+1} \max_{0\le N\le N_{j_0}}
\big\|K_N\big\|_{\infty}\\
&\le 2.
\end{align*}
Combining this with \eqref{eqn:sup>L}, we find that
\begin{align*}
\max_{0\le N\le N_{j_0}} 
\big|K_N*f(x)\big|
&\ge \max_{0\le N\le N_{j_0}} |II| - \max_{0\le N\le N_{j_0}} |I|-\max_{0\le N\le N_{j_0}} |III|\\
&\ge j_0-2.
\end{align*}
Since $j_0$ can be chosen arbitrarily large, we obtain
$$\sup_{N\ge 0} 
\big|K_N*f(x)\big|=\infty.$$
This completes the proof of Theorem \ref{thm:cesaro-divergence} since $\big|K_N*f(x)\big|<\infty$ for all $N$. 
\end{proof}

It remains to prove Lemma \ref{lem:tilde-K-mu}. Let $r>0$. We will call $\{y_j\}_{j=1}^m\subset\mathbb S^n$ an $r$-separated set if 
$$|y_j-y_{j'}|\ge r\ \text{whenever } j\neq j'.$$
We call $\{y_j\}_{j=1}^m$ a maximal $r$-separated set if it is not strictly contained in another $r$-separated set. 
By successively adding points to a singleton, it is easy to see that maximal $r$-separated sets exist for all $r>0$; moreover, when $0<r\le\pi$, by \eqref{eq:AD-regularity}, the corresponding cardinality $m$ must satisfy 
$$r^{-n}/C\le m\le C r^{-n}$$ 
for some constant $C>1$. 

\begin{lemma}\label{lem:Riemann-sum}
Let $\{y_j\}_{j=1}^m\subset\mathbb S^n$ be a maximal $r$-separated set with $0<r\le\pi$. Then
\begin{equation}\label{eq:log-lower-bound}
\frac{1}{m}\sum_{j=1}^m \frac{1}{|x-y_j|^n}\ge C \log(\pi/r),\ \forall x\in \mathbb S^n.
\end{equation}
\end{lemma}

\begin{proof}
The proof follows easily from a $C$-adic (with $C>1$ sufficiently large) decomposition around $x$, and the mass distribution principle (for the latter, cf. \cite[Theorem~5.7]{Mattila1995}).
\end{proof}

Let $\{y_j\}_{j=1}^m\subset\mathbb S^n$ satisfy \eqref{eq:log-lower-bound}. By a small perturbation of $\{y_j\}_{j=1}^m$ we may assume that $\{y_j\}_{j=1}^m$ contains no antipodal pairs.

\begin{lemma}\label{lem:point-masses}
Let $\{y_j\}_{j=1}^m\subset\mathbb S^n$ be a set of $m$ distinct points which contains no antipodal pairs. Then the following hold.\\
(i) For almost every $x\in\mathbb S^n$, the numbers 
$$\pi, |x-y_1|, \cdots, |x-y_m|$$ 
are rationally independent;\\
(ii) If moreover $\{y_j\}_{j=1}^m$ satisfies \eqref{eq:log-lower-bound}, then, with $\mu=\frac{1}{m} \sum_{j=1}^m \delta_{y_j}$, we have $$\limsup_{N\rightarrow\infty} 
\big|\widetilde K_N^{(0)}*\mu(x)\big|\ge C \log(\pi/r) - \widetilde C,\ \text{a.e. }x\in\mathbb S^n,$$
where $\widetilde C>0$ is a constant depending only on $n$. 
\end{lemma}

\begin{proof}
(i) It suffices to show that, for any $q_0, q_1, \cdots, q_m\in\mathbb Q$ not all of which are zero, letting
$$F(x)=q_0 \pi + \sum_{j=1}^m q_j |x-y_j|,$$ 
the zero set 
$$Z=\left \{x\in\mathbb S^n: F(x)=0\right \}$$
has measure zero. Without loss of generality, we may assume that $q_1, \cdots, q_m$ are not all equal to zero, since otherwise $F(x)\equiv q_0 \pi\neq 0$. Notice that $F(x)$ is real-analytic on 
$$U=\mathbb S^n\backslash\big (\{y_j\}_{j=1}^m\cup\{\hat y_j\}_{j=1}^m\big),$$
and is not constantly zero on $U$ (since $F(x)$ is nondifferentiable at the $y_j$'s for which $q_j\neq 0$). It follows that $Z$ must have measure zero (cf. \cite[Chapter VII, Lemma~4.17]{SteinWeiss1971}; see also \cite{Mityagin2015}).

(ii) Note that
$$\widetilde K_N^{(0)}*\mu(x)=\frac{1}{m} \sum_{j=1}^m \widetilde K_N^{(0)}(x,y_j).$$
By \eqref{eq:tilde-K-0}, \eqref{eq:C_N-asympt}, and Lemma \ref{lem:jacobi-poly-asympt}, for $x\in U$ we have
\begin{align*}
\widetilde K_N^{(0)}*\mu(x)
&=\frac{\theta_N}{m} \sum_{j=1}^m k^{(0)}(x,y_j)
\cos\left (\Big(N+\frac{3n-1}{4}\Big)|x-y_j|-\frac{n\pi}{2}\right ) + O(N^{-1}), 
\end{align*}
where $\theta_N={C_N}{N^{-1/2}}$ satisfies \eqref{eq:C_N-asympt}, and 
$$k^{(0)}(x,y)=\mathds 1_{\{|x-y|\le \pi/2\}}
k(x,y)$$
(with $k(x,y)$ given by \eqref{eq:k(x,y)}). By (i), for almost every $x\in U$, the numbers $2\pi, |x-y_1|, \cdots, |x-y_m|$ are  rationally independent. It follows by Kronecker's
theorem (cf. \cite[Chapter~23]{HardyWright2008}) that, for such $x$, 
$$\limsup_{N\rightarrow\infty} 
\big|\widetilde K_N^{(0)}*\mu(x)\big|=\frac{C}{m} \sum_{j=1}^m k^{(0)}(x,y_j).$$
On the other hand, by \eqref{eq:k(x,y)},
\begin{align*}
\frac{C}{m} \sum_{j=1}^m k^{(0)}(x,y_j)
& \ge \frac{C}{m} \sum_{j=1}^m \mathds 1_{\{|x-y_j|\le \pi/2\}}\frac{1}{|x-y_j|^n}\\
& = \frac{C}{m} \sum_{j=1}^m \frac{1}{|x-y_j|^n}
-\frac{C}{m} \sum_{j=1}^m \mathds 1_{\{|x-y_j|> \pi/2\}}\frac{1}{|x-y_j|^n}\\
& \ge \frac{C}{m} \sum_{j=1}^m \frac{1}{|x-y_j|^n}
-\widetilde C.
\end{align*}
Combining this with the assumption that $\{y_j\}_{j=1}^m$ satisfies \eqref{eq:log-lower-bound}, the proof of (ii) is complete.
\end{proof}

\begin{proof}[Proof of Lemma \ref{lem:tilde-K-mu}]	
Lemma \ref{lem:tilde-K-mu} now follows immediately by taking 
$$\mu=\frac{1}{m} \sum_{j=1}^m \delta_{y_j},$$
where $\{y_j\}_{j=1}^m$ is as in Lemma \ref{lem:point-masses}\,(ii), with $r>0$ chosen small enough that 
$$C \log(\pi/r) - \widetilde C> L.$$
\end{proof}

\section{Equidivergence of summation methods}\label{sec:summation}

In this section, we use Theorem \ref{thm:cesaro-divergence} to obtain almost everywhere divergence results for Riesz and Bochner-Riesz means. 

For $\delta>0$, the Riesz $(R,\delta)$ means of $f\in L^1(\mathbb S^n)$ are defined by
\begin{align}
\widetilde S^\delta_R f 
= \sum_{k<R} \left (1-\frac{k}{R}\right )^\delta \text{proj}_k f,\ R>0.\label{eq:riesz-means-def}
\end{align}
Equivalently,
\begin{equation}\label{eq:riesz-means-def-2}
\widetilde S^\delta_R f 
= \frac{1}{R^\delta}\sum_{k<R} \left (R-{k}\right )^\delta \text{proj}_k f.
\end{equation}

We first show that Theorem \ref{thm:cesaro-divergence} implies the following. 

\begin{corollary}\label{thm:riesz-divergence}
There exists a function $f\in L^1(\mathbb S^n)$ such that
$$\limsup_{R\rightarrow\infty}\big|\widetilde S_R^{\frac{n-1}{2}} f(x)\big|=\infty,\ \text{a.e. } x\in\mathbb S^n.$$
\end{corollary}

To prove Corollary \ref{thm:riesz-divergence}, we will use the following result of Ingham \cite{Ingham1968/1969}.

\begin{lemma}[Theorem B of \cite{Ingham1968/1969}, a special case]\label{lem:Ingham-B}
Let $\delta>0$ and let $m=\lceil \delta\rceil+2$. There exist constants $c_1,\cdots,c_m$, such that
$$\binom{k+\delta}{k}=\sum_{j=1}^m c_j \left (k+\frac{j}{m}\right )^{\delta} + O\big((k+1)^{-2}\big),\ k=0,1,\cdots.$$
\end{lemma}

Note that Lemma \ref{lem:Ingham-B} implies
\begin{align*}
S_N^\delta f
=& \sum_{j=1}^m c_j \frac{(N+\frac{j}{m})^\delta}{A_N^\delta} \widetilde S^\delta_{N+\frac{j}{m}} f 
+\frac{1}{A_N^\delta} \sum_{k=0}^N O\big((N-k+1)^{-2}\big) \text{proj}_k f \\
=& O(1)\max_{1\le j\le m} \big|\widetilde S^\delta_{N+\frac{j}{m}} f\big| +O(1)\max_{k\le N} \frac{|\text{proj}_k f|}{(k+1)^\delta}.
\end{align*}
In particular, taking $\delta=\delta_0$, we obtain the following relation between Ces\`aro and Riesz means. 

\begin{lemma}\label{lem:riesz-cesaro}
There exists a constant $C>0$, such that for any $f\in L^1(\mathbb S^n)$, 
$$\limsup_{N\rightarrow\infty} \big|S_N^{\frac{n-1}{2}} f(x)\big| \le C \limsup_{R\rightarrow\infty} \big|\widetilde S_R^{\frac{n-1}{2}} f(x)\big| + C \sup_{k\ge 0}\frac{|\emph{proj}_k f(x)|}{(k+1)^{\frac{n-1}{2}}},\ \forall x\in\mathbb S^n.$$
\end{lemma}

By Lemma \ref{lem:riesz-cesaro}, it follows that
$$\limsup_{R\rightarrow\infty}\big|\widetilde S_R^{\frac{n-1}{2}} f(x)\big|=\infty$$
provided
$$\limsup_{N\rightarrow\infty}\big|S_N^{\frac{n-1}{2}} f(x)\big|=\infty
\quad\text{and}\quad\sup_{k\ge 0}\frac{|\text{proj}_k f(x)|}{(k+1)^{\frac{n-1}{2}}}<\infty.$$
Thus, in view of Theorem \ref{thm:cesaro-divergence}, to prove Corollary \ref{thm:riesz-divergence} it remains to show the following. 

\begin{lemma}\label{lem:maximal-projection}
There exists a constant $C>0$, such that for any $f\in L^1(\mathbb S^n)$, 
\begin{equation}\label{eq:proj-L1-bound}
\left \|\sup_{k\ge 0}\frac{|\emph{proj}_k f|}{(k+1)^{\frac{n-1}{2}}}\right \|_{L^1(\mathbb S^n)} \le C \|f\|_{L^1(\mathbb S^n)}.
\end{equation}
\end{lemma}

\begin{proof}
By \eqref{eq:Z-bound}, we have
$$\frac{|\text{proj}_k f|}{(k+1)^{\frac{n-1}{2}}}
\le C k^{(0,\pi)}*|f|,\ k=0,1,\cdots,$$
where
$$k^{(0,\pi)}(x,y)= |x-y|^{-\frac{n-1}{2}} |x-\hat y|^{-\frac{n-1}{2}}.$$
Since
$$\int_{\mathbb S^n} k^{(0,\pi)}(x,y)dx \equiv C<\infty,$$
\eqref{eq:proj-L1-bound} follows immediately from Fubini's theorem.
\end{proof}

Let $c\in\mathbb R$. We define the shifted Riesz $(R,\delta)$ means by
\begin{equation}\label{eq:shifted-riesz-def}
\widetilde S^{\delta,c}_R f 
= \sum_{k+c<R} \left (1-\frac{k+c}{R}\right )^\delta \text{proj}_k f,\ R>0.
\end{equation}
Note that, when $R>c$, we can write
\begin{align}
\widetilde S^{\delta,c}_R f 
&= \left (1-\frac{c}{R}\right )^\delta \sum_{k<R-c} \left (1-\frac{k}{R-c}\right )^\delta \text{proj}_k f \notag\\
&= \left (1-\frac{c}{R}\right )^\delta \widetilde S^{\delta}_{R-c} f.\label{eq:shifted-identity}
\end{align}
Since $\left (1-\frac{c}{R}\right )^\delta\rightarrow1$ as $R\rightarrow\infty$, Corollary \ref{thm:riesz-divergence} implies the following.

\begin{corollary}\label{thm:riesz-divergence-shifted}
For any $c\in\mathbb R$, there exists a function $f\in L^1(\mathbb S^n)$ such that
$$\limsup_{R\rightarrow\infty}\big|\widetilde S_R^{\frac{n-1}{2},c} f(x)\big|=\infty,\ \text{a.e. } x\in\mathbb S^n.$$
\end{corollary}

To handle Bochner-Riesz means, we will need another result of Ingham from \cite{Ingham1968/1969}.

\begin{lemma}[Theorem A of \cite{Ingham1968/1969}, a special case]\label{lem:Ingham-A}
Let $\delta>0$ and let $m=\lceil \delta\rceil+1$. There exist polynomials $p_0,\cdots,p_m$, such that 
$$(k+\varepsilon)^\delta 
= \sum_{j=0}^{m} p_j(\varepsilon) A^\delta_{k-j} 
+ O\big((k+1)^{-2}\big),\ 0<\varepsilon\le 1,\ k=0,1,\cdots,$$ 
where $A^\delta_{k}:=0$ if $k<0$.
\end{lemma}

Lemma \ref{lem:Ingham-A} implies that, for $N\ge 1$,
\begin{align*}
\widetilde S^\delta_{N+\varepsilon} f
& = \sum_{j=0}^{m\wedge N} p_j(\varepsilon) \frac{A_{N-j}^\delta}{(N+\varepsilon)^\delta} S_{N-j}^\delta f
+ \frac{1}{(N+\varepsilon)^\delta}\sum_{k=0}^N O\big((N-k+1)^{-2}\big) \text{proj}_k f\\
& = O(1)\max_{0\le j\le m\wedge N}\big|S_{N-j}^\delta f\big|
+ {O(1)} \max_{k\le N} \frac{|\text{proj}_k f|}{(k+1)^\delta}.
\end{align*}
From this we obtain the following converse to Lemma \ref{lem:riesz-cesaro}.

\begin{lemma}\label{lem:cesaro-riesz}
For any $\delta>0$, there exists a constant $C>0$, such that for any $f\in L^1(\mathbb S^n)$, 
$$\sup_{R>0} \big|\widetilde S_R^{\delta} f(x)\big|
\le C \sup_{N\ge 0} \big|S_N^{\delta} f(x)\big|
+ C \sup_{k\ge 0}\frac{|\emph{proj}_k f(x)|}{(k+1)^{\delta}},\ \forall x\in\mathbb S^n.$$
\end{lemma}

For later application, we record that, when $\delta=\delta_0+1$, it follows from Lemmas \ref{lem:antipodal}, \ref{lem:major-a}, and \ref{lem:K^{delta+1}} that for any $f\in L^1(\mathbb S^n)$, 
$$\sup_{N\ge 0} \big|S_N^{\delta_0+1} f(x)\big|<\infty,\ \text{a.e. }x\in \mathbb S^n.$$
Combining this with Lemma \ref{lem:cesaro-riesz}, Lemma \ref{lem:maximal-projection}, and \eqref{eq:shifted-identity}, we see that for any $f\in L^1(\mathbb S^n)$ and $c\in\mathbb R$,
\begin{equation}\label{eq:delta0+1}
\sup_{R>0} \big|\widetilde S_R^{\delta_0+1,c} f(x)\big|<\infty,\ \text{a.e. }x\in \mathbb S^n.
\end{equation}

Now let $c\ge0$. Consider the shifted, quadratic, Riesz $(R,\delta)$ means \begin{equation}\label{eq:quadratic-riesz-means-def}
B^{\delta,c}_R f 
= \sum_{k+c<R} \left (1-\frac{(k+c)^2}{R^2}\right )^\delta \text{proj}_k f,\ R>0.
\end{equation}
Notice that for any $t\in [0,1)$, 
\begin{align*}
(1-t^2)^\delta
&=(1-t)^\delta(1+t)^\delta \\
&=(1-t)^\delta\big(2-(1-t)\big)^\delta \\
&=2^\delta (1-t)^\delta + 2^\delta \sum_{\ell=1}^\infty \binom{\delta}{\ell}\frac{(-1)^\ell}{2^\ell} (1-t)^{\delta+\ell}.
\end{align*}
This allows us to write
\begin{align}
B^{\delta,c}_R f 
&= 2^\delta \widetilde S_R^{\delta,c} f + 2^\delta \sum_{\ell=1}^\infty  \binom{\delta}{\ell}\frac{(-1)^\ell}{2^\ell} \widetilde S_R^{\delta+\ell,c} f.\label{eq:BR-Riesz}
\end{align}
We will bound the series on the right-hand side by its first term. For this we will use the following lemma (see also \cite[p.~279]{SteinWeiss1971} for a related result).

\begin{lemma}\label{lem:major-b}
Let $\psi\in L^1[0,1]$ with $\psi\ge 0$, and let
\begin{equation}\label{eq:phi-psi}
\varphi(t)=\int_0^t \psi(t-s)s^{\delta}ds,\ 0\le t\le 1.
\end{equation}
Let $\{a_k\}_{k=0}^{\infty}$ be a sequence of numbers. For $R>0$, let
\begin{align*}
&\widetilde S^{\delta,c}_R 
= \sum_{k+c<R} \left (1-\frac{k+c}{R}\right )^\delta a_k,\\
&\widetilde S^{\varphi,c}_R 
= \sum_{k+c<R} \varphi\left (1-\frac{k+c}{R}\right ) a_k.
\end{align*}
Then for any $R>0$, 
$$\big|\widetilde S^{\varphi,c}_R \big| 
\le \varphi(1) \sup_{r\le R} \big|\widetilde S^{\delta,c}_r \big|.$$
\end{lemma}

\begin{proof}
Notice that, for any $t\in [0,1]$,
\begin{align*}
\varphi(1-t)
&=\int_0^{1-t} \psi(1-t-s)s^{\delta}ds\\
&=\int_0^{1} \psi(1-u)(u-t)_+^{\delta}du\\
&=\int_0^{1} \psi(1-u)u^\delta \left (1-\frac{t}{u}\right )_+^{\delta}du.
\end{align*}
Substituting $t=\frac{k+c}{R}$, we find that
\begin{align*}
\widetilde S^{\varphi,c}_R 
&= \int_0^{1} \psi(1-u)u^\delta \widetilde S^{\delta,c}_{uR} du.
\end{align*}
Now bounding 
$$\big|\widetilde S^{\delta,c}_{uR}\big|\le \sup_{r\le R} \big|\widetilde S^{\delta,c}_r \big|,\ 0<u\le 1$$
and noting that
$$\int_0^{1} \psi(1-u)u^\delta du=\varphi(1),$$
the desired bound follows.
\end{proof}

Note that, for any $\rho>0$, $\varphi(t)=t^{\delta+\rho}$ satisfies \eqref{eq:phi-psi} with 
$$\psi(t)=\frac{\Gamma(\delta+\rho+1)}{\Gamma(\delta+1)\Gamma(\rho)} t^{\rho-1}.$$
Thus, by Lemma \ref{lem:major-b}, for any $\rho>0$ we have
$$\big|\widetilde S^{\delta+\rho,c}_R\big|
\le \sup_{r\le R}\big|\widetilde S^{\delta,c}_r\big|,\ R>0.$$
Applying this with $\rho=1,2,\cdots$, we see from \eqref{eq:BR-Riesz} that
\begin{equation}\label{eq:BR-Riesz-2}
B^{\delta,c}_R f 
= 2^\delta \widetilde S_R^{\delta,c} f + O(1)\sup_{r\le R}\big|\widetilde S^{\delta+1,c}_r f\big|,\ R>0.
\end{equation}
Letting $\delta=\delta_0$, and combining \eqref{eq:delta0+1}, \eqref{eq:BR-Riesz-2}, and Corollary \ref{thm:riesz-divergence-shifted}, we obtain the following.

\begin{corollary}\label{thm:BR-divergence-shifted}
For any $c\ge 0$, there exists a function $f\in L^1(\mathbb S^n)$ such that
$$\limsup_{R\rightarrow\infty}\big|B^{\frac{n-1}{2},c}_R f(x)\big|=\infty,\ \text{a.e. } x\in\mathbb S^n.$$
\end{corollary}

Finally, recall that the Bochner-Riesz means of order $\delta$ are defined by
\begin{align*}
B^{\delta}_R f 
= \sum_{k(k+n-1)<R^2} \left (1-\frac{k(k+n-1)}{R^2}\right )^\delta \text{proj}_k f,\ R>0, 
\end{align*}
where $-k(k+n-1)$ is the eigenvalue in \eqref{eq:eigenvalue}. Notice that, writing $c=\frac{n-1}{2}$, 
\begin{align*}
B^{\delta}_R f 
&= \sum_{k(k+n-1)<R^2} \left (1-\frac{(k+c)^2-c^2}{R^2}\right )^\delta \text{proj}_k f\\
&= \left (1+\frac{c^2}{R^2}\right )^\delta 
\sum_{(k+c)^2<R^2+c^2} \left (1-\frac{(k+c)^2}{R^2+c^2}\right )^\delta  \text{proj}_k f\\
&= \left (1+\frac{c^2}{R^2}\right )^\delta 
B^{\delta,c}_{\sqrt{R^2+c^2}} f.
\end{align*}
Thus, from Corollary \ref{thm:BR-divergence-shifted} we also obtain almost everywhere divergence of the Bochner-Riesz means. 

\begin{corollary}\label{thm:BR-divergence}
There exists a function $f\in L^1(\mathbb S^n)$ such that
$$\limsup_{R\rightarrow\infty}\big|B^{\frac{n-1}{2}}_R f(x)\big|=\infty,\ \text{a.e. } x\in\mathbb S^n.$$
\end{corollary}

\noindent\textit{Acknowledgment.} We would like to thank Stephen Wainger and Sergey Denisov for the suggestion of using a deterministic construction. We also thank Andreas Seeger and Jongchon Kim for enlightening discussions on related topics. 

\bibliographystyle{abbrv}
\bibliography{bibliography}
\nocite{}

\end{document}